\documentclass[11pt]{article}
\usepackage{amssymb,amsmath,amsthm}
\usepackage{palatino}
\usepackage{url}
\usepackage[margin=3cm,letterpaper]{geometry}

\title{On orthogonal matrices with zero diagonal}
\author{R.~F.~Bailey\footnote{School of Science and the Environment (Mathematics), Grenfell Campus, Memorial University of Newfoundland, Corner Brook, NL A2H~6P9, Canada. E-mail: \texttt{rbailey@grenfell.mun.ca}} \and 
R.~Craigen\footnote{Department of Mathematics, University of Manitoba, Winnipeg, MB R3T~2N2, Canada. E-mail: \texttt{robert.craigen@umanitoba.ca}}}

\newtheorem{theorem}{Theorem}[section]
\newtheorem{lemma}[theorem]{Lemma}
\newtheorem{proposition}[theorem]{Proposition}
\newtheorem{corollary}[theorem]{Corollary}

\theoremstyle{definition}
\newtheorem{defn}[theorem]{Definition}
\newtheorem{example}[theorem]{Example}

\newcommand{\reals}{\mathbb{R}}

\newcommand{\DRT}{\mathrm{DRT}}
\newcommand{\OMZD}{\mathrm{OMZD}}
\newcommand{\OMPZD}{\mathrm{OMPZD}}

\begin{document}

\maketitle

\begin{abstract}
We consider real orthogonal $n\times n$ matrices whose diagonal entries are zero and off-diagonal entries nonzero, which we refer to as $\OMZD(n)$.  We show that there exists an $\OMZD(n)$ if and only if $n\neq 1,\ 3$, and that a symmetric $\OMZD(n)$ exists if and only if $n$ is even and $n\neq 4$.  We also give a construction of $\OMZD(n)$ obtained from doubly regular tournaments.  Finally, we apply our results to determine the minimum number of distinct eigenvalues of matrices associated with some families of graphs, and consider the related notion of orthogonal matrices with partially-zero diagonal.\\


\hrule
\end{abstract}

\section{Introduction} \label{sec:intro}
An $n\times n$ real matrix $A$ is {\em orthogonal} if and only if $AA^\mathrm{T} = cI$ for some positive constant $c\in\reals^+$.  (We do not insist that $c=1$, i.e.\ that $AA^\mathrm{T}=I$.)  A number of papers, such as~\cite{Beasley93,Cheon99,Deaett11,Larson13}, have studied the {\em pattern} (or the {\em support}) of a matrix, i.e.\ the arrangement of its zero and nonzero entries, and which patterns admit matrices with given algebraic properties (for instance, orthogonal, positive definite, etc.).  In this vein, a pattern is called an {\em orthogonal pattern} if it admits an orthogonal matrix.  A seemingly natural pattern to consider is where the zero entries are precisely those on the main diagonal; orthogonal matrices with this pattern are the subject of this paper.  For brevity, we make the following definition.

\begin{defn} \label{defn:omzd}
Let $A$ be an $n\times n$ real matrix.  We say that $A$ is an {\em orthogonal matrix with zero diagonal}, or an $\OMZD(n)$, if and only if it is orthogonal, its diagonal entries are all zero, and its off-diagonal entries are all nonzero.
\end{defn}

A matrix where all diagonal entries are zero is sometimes called a {\em hollow} matrix (see \cite[Section~3.1.3]{Gentle17}, for instance); we will not use this term as it also allows for off-diagonal zero entries.
A {\em conference matrix} of order $n$ is an $\OMZD(n)$ whose off-diagonal entries are all $\pm 1$. 
(For further details see \mbox{\cite[Section~V.6]{handbook}}, \cite{Goethals67} or \cite[Section~4.3]{Ionin06}.)
\begin{example} \label{example:conf} \rm
The following are conference matrices of orders $2$, $4$ and $6$ (taken from~\cite[Example~V.6.4]{handbook}):
\[ \begin{bmatrix} 0 & 1 \\ 1 & 0 \end{bmatrix},
   \qquad
   \begin{bmatrix} 0 &  1 &  1 &  1 \\
                  -1 &  0 & -1 &  1 \\
                  -1 &  1 &  0 & -1 \\
                  -1 & -1 &  1 &  0 \end{bmatrix},
\qquad
\begin{bmatrix} 0 &  1 &  1 &  1 &  1 &  1 \\
                1 &  0 &  1 & -1 & -1 &  1 \\
								1 &  1 &  0 &  1 & -1 & -1 \\
								1 & -1 &  1 &  0 &  1 & -1 \\
								1 & -1 & -1 &  1 &  0 &  1 \\
								1 &  1 & -1 & -1 &  1 &  0	\end{bmatrix}. \]
\end{example}

There are known necessary conditions for the existence of conference matrices, and owing to the constraint on their entries it is not surprising that they are difficult, in general, to construct.  From the definition one can show that $n$ must be even.  Two conference matrices are {\em equivalent} if one can be obtained from the other by multiplying rows or columns by $-1$, or by simultaneously permuting rows and columns; if $n\equiv 2$~mod~$4$, a conference matrix is equivalent to a symmetric matrix, while if $n \equiv 0$~mod~$4$, a conference matrix is equivalent to a skew-symmetric matrix (see, for instance,~\cite{DGS71} or~\cite{Craigen91}).  However, $n$ being even is not a sufficient condition: Belevitch~\cite{Belevitch50} showed that if a symmetric conference matrix of order~$n$ exists (and thus $n\equiv 2$~mod~$4$), then $n-1$ must be the sum of two squares; consequently, there do not exist conference matrices of orders $22$, $34$ or $58$.  It is unknown if there exists a conference matrix of order $66$ (see \mbox{\cite[Table~V.6.10]{handbook}}.  However, the following result, implied by the main construction of Paley's classic 1933 paper~\cite{Paley33}, gives an infinite family of examples.

\begin{theorem} \label{theorem:paleyconf}
Let $q$ be an odd prime power.  If $q\equiv 1$~mod~$4$ then there exists a symmetric conference matrix of order $q+1$; if $q\equiv 3$~mod~$4$ then there exists a skew-symmetric conference matrix of order $q+1$.
\end{theorem}

The purpose of this article is to investigate what happens when the requirement that the nonzero entries are $\pm 1$ is removed.  For odd orders, it is not difficult to see that neither an $\OMZD(1)$ nor an $\OMZD(3)$ can exist.  However, as the following two examples demonstrate, OMZDs with odd orders {\em do} exist.

\begin{example} \label{example:5x5} \rm
The following matrix is an $\OMZD(5)$:
\[ \begin{bmatrix} 0 & 1 & 1 & 1 & 1 \\
                   1 & 0 & a & 1 & b \\
                   1 & b & 0 & a & 1 \\
                   1 & 1 & b & 0 & a \\
                   1 & a & 1 & b & 0 \end{bmatrix} \]
where $a=\frac{-1+\sqrt{3}}{2}$ and $b=\frac{-1-\sqrt{3}}{2}$.
\end{example}

\newpage
\begin{example} \label{example:7x7} \rm
The following matrix is an $\OMZD(7)$ (found with the help of \textsc{Maple}):
\[ \begin{bmatrix} 0 & 1 & 1 & 1 & 1 & 1 & 1 \\
                   1 & 0 & a & b & c & 1 & d \\
                   1 & d & 0 & a & b & c & 1 \\
                   1 & 1 & d & 0 & a & b & c \\
                   1 & c & 1 & d & 0 & a & b \\
                   1 & b & c & 1 & d & 0 & a \\
                   1 & a & b & c & 1 & d & 0 \end{bmatrix} \]
where $a= -\frac{1}{\sqrt{6}}(1-R)-\frac{1}{3}R$, $b= -1+\frac{\sqrt{6}}{2}$, $c= -\frac{1}{\sqrt{6}}(1+R)+\frac{2}{3}R$ and $d= -\frac{1}{\sqrt{6}}-\frac{1}{3}R$, and where $R=\sqrt{9+4\sqrt{6}}$.
\end{example}

The paper is organized as follows.  In Sections~\ref{sec:existence} and~\ref{sec:symmetric}, we prove that there exists an $\OMZD(n)$ if and only if $n\not\in\{1,3\}$, and also consider the existence of symmetric $\OMZD(n)$, proving that these exist if and only if $n$ is even and $n\neq 4$.  In Section~\ref{sec:drt} we give an alternative, combinatorial construction of $\OMZD(n)$ for certain $n\equiv 3$~mod~$4$.  Finally, in Section~\ref{sec:eigenvalues} we give some applications of OMZDs to spectral graph theory, to the problem of determining the minimum number of distinct eigenvalues for a class of real, symmetric matrices obtained from a bipartite graph.  This also leads us to consider the existence of orthogonal matrices with partially-zero diagonal (OMPZDs).

\section{Existence of OMZDs} \label{sec:existence}

To demonstrate the existence of OMZDs, we make use of the following construction.

\begin{lemma} \label{lemma:adding}
Suppose that there exist an $\OMZD(m+1)$ and an $\OMZD(n+1)$.  Then there exists an $\OMZD(m+n)$.
\end{lemma}

\begin{proof} Let $M$ and $N$ be an $\OMZD(m+1)$ and an $\OMZD(n+1)$ respectively, each scaled so that $MM^\mathrm{T}=I_{m+1}$ and $NN^\mathrm{T}=I_{n+1}$.  We express these in the following form:
\renewcommand{\arraystretch}{1.25}
\[ M = \left[ \begin{array}{c|c} 0 & u^\mathrm{T} \\ \hline v & B \end{array} \right], \qquad N = \left[ \begin{array}{c|c} 0 & x^\mathrm{T} \\ \hline y & C \end{array} \right]
\]%
where $B$ and $C$ are $m\times m$ and $n\times n$, respectively (known as the {\em cores} of $M$ and $N$).  By construction, we have that
\[ MM^\mathrm{T} = \left[ \begin{array}{c|c} u^\mathrm{T}u & u^\mathrm{T}B^\mathrm{T} \\ \hline Bu & vv^\mathrm{T} + BB^\mathrm{T} \end{array} \right] = \left[ \begin{array}{c|c} 1 & 0 \\ \hline 0 & I_m \end{array} \right] \]
and
\[ NN^\mathrm{T} = \left[ \begin{array}{c|c} x^\mathrm{T}x & x^\mathrm{T}C^\mathrm{T} \\ \hline Cx & yy^\mathrm{T} + CC^\mathrm{T} \end{array} \right] = \left[ \begin{array}{c|c} 1 & 0 \\ \hline 0 & I_n \end{array} \right]. \]
Now let
\[ Q = \left[ \begin{array}{c|c} B & vx^\mathrm{T} \\ \hline yu^\mathrm{T} & C \end{array} \right]. \]
Then
\[ QQ^\mathrm{T} = \left[ \begin{array}{c|c} BB^\mathrm{T} + vx^\mathrm{T}xv^\mathrm{T} & Buy^\mathrm{T} + vx^\mathrm{T}C^\mathrm{T} 
                                   \\ \hline yu^\mathrm{T}B^\mathrm{T} + Cxv^\mathrm{T} & yu^\mathrm{T}uy^\mathrm{T} + CC^\mathrm{T} \end{array} \right] 
								 = \left[ \begin{array}{c|c} I_m & 0 \\ \hline 0 & I_n \end{array} \right] = I_{m+n}, \]
\renewcommand{\arraystretch}{1.0}%
so $Q$ is orthogonal.  By construction, $Q$ has all diagonal entries zero and all off-diagonal entries nonzero.  Therefore $Q$ is an $\OMZD(m+n)$.
\end{proof}

As an example of the usefulness of Lemma~\ref{lemma:adding}, we construct an $\OMZD(9)$ from the $\OMZD(6)$ and $\OMZD(5)$ we saw in Examples~\ref{example:conf} and~\ref{example:5x5} respectively.

\begin{example} \label{example:omzd9} \rm
After rescaling, we obtain the following OMZDs from Examples \ref{example:conf} and~\ref{example:5x5}:
\[ \frac{1}{\sqrt{5}} \begin{bmatrix} 0 &  1 &  1 &  1 &  1 &  1 \\
                1 &  0 &  1 & -1 & -1 &  1 \\
								1 &  1 &  0 &  1 & -1 & -1 \\
								1 & -1 &  1 &  0 &  1 & -1 \\
								1 & -1 & -1 &  1 &  0 &  1 \\
								1 &  1 & -1 & -1 &  1 &  0	\end{bmatrix},
\qquad
\frac{1}{2} \begin{bmatrix} 0 & 1 & 1 & 1 & 1 \\
                               1 & 0 & a & 1 & b \\
                               1 & b & 0 & a & 1 \\
                               1 & 1 & b & 0 & a \\
                               1 & a & 1 & b & 0 \end{bmatrix}
\]
(where $a=\frac{-1+\sqrt{3}}{2}$ and $b=\frac{-1-\sqrt{3}}{2}$).  Using the notation in the proof of Lemma~\ref{lemma:adding}, we have
\[ B = \frac{1}{\sqrt{5}} \begin{bmatrix} 0 &  1 & -1 & -1 &  1 \\
								                          1 &  0 &  1 & -1 & -1 \\
							                           -1 &  1 &  0 &  1 & -1 \\
								                         -1 & -1 &  1 &  0 &  1 \\
								                          1 & -1 & -1 &  1 &  0	\end{bmatrix},
\qquad
   C = \frac{1}{2} \begin{bmatrix} 0 & a & 1 & b \\
                                   b & 0 & a & 1 \\
                                   1 & b & 0 & a \\
                                   a & 1 & b & 0 \end{bmatrix}.
\]
By construction, we have $vx^\mathrm{T} = \frac{1}{2\sqrt{5}} K$ and $yu^\mathrm{T} = \frac{1}{2\sqrt{5}} K^\mathrm{T}$ (where $K$ denotes the $5\times 4$ all-ones matrix).  After rescaling by a factor of $2\sqrt{5}$, we obtain the following $\OMZD(9)$:
\renewcommand{\arraystretch}{1.2}
\[ \left[ \begin{array}{ccccc|cccc}
 0 &  2 & -2 & -2 &  2 & 1 & 1 & 1 & 1 \\
 2 &  0 &  2 & -2 & -2 & 1 & 1 & 1 & 1 \\
-2 &  2 &  0 &  2 & -2 & 1 & 1 & 1 & 1 \\
-2 & -2 &  2 &  0 &  2 & 1 & 1 & 1 & 1 \\
 2 & -2 & -2 &  2 &  0 & 1 & 1 & 1 & 1 \\
\hline
 1 &  1 &  1 &  1 &  1 & 0 & a\sqrt{5} & \sqrt{5} & b\sqrt{5} \\
 1 &  1 &  1 &  1 &  1 & b\sqrt{5} & 0 & a\sqrt{5} & \sqrt{5} \\
 1 &  1 &  1 &  1 &  1 & \sqrt{5} & b\sqrt{5} & 0 & a\sqrt{5} \\
 1 &  1 &  1 &  1 &  1 & a\sqrt{5} & \sqrt{5} & b\sqrt{5} & 0 
\end{array} \right]. \]
\renewcommand{\arraystretch}{1.0}
\end{example}

The main result of this paper is the following.

\begin{theorem} \label{theorem:existence}
There exists an $\OMZD(n)$ if and only if $n\not\in\{1,3\}$.
\end{theorem}

\begin{proof} We proceed by induction on $n$.  As observed earlier, an $\OMZD(1)$ and $\OMZD(3)$ do not exist, while $\OMZD(n)$ for $n=2,4,5$ are given in Examples~\ref{example:conf} and~\ref{example:5x5}.  Assume that $n\geq 6$, and suppose that there exist $\OMZD(m)$ for all $m<n$ (except $m=1,3$).  In particular, there exist an $\OMZD(n-2)$ and an $\OMZD(4)$, so by Lemma~\ref{lemma:adding} there exists an $\OMZD(n)$.  The result follows by induction.
\end{proof}

\section{Existence of symmetric OMZDs} \label{sec:symmetric}

In this section, we settle the existence question for symmetric OMZDs.  In Example~\ref{example:conf}, we have already seen a symmetric $\OMZD(2)$; however, for the next possible order we are less fortunate. 

\begin{proposition} \label{prop:symm4}
A symmetric $\OMZD(4)$ does not exist.
\end{proposition}

\begin{proof} We give a proof by contradiction.  Let $M$ be a symmetric $\OMZD(4)$, which must have the following form:
\[ M = \begin{bmatrix}
   0 & a & b & c \\
	 a & 0 & d & e \\
	 b & d & 0 & f \\
	 c & e & f & 0 \end{bmatrix} \]
where $a,b,c,d,e,f$ are all nonzero.  By the definition of orthogonality the row norms of $M$ must all be equal, so the first two rows give $a^2+b^2+c^2 = a^2+d^2+e^2$, and thus $b^2+c^2 = d^2+e^2$.  Similarly, the remaining two rows give $b^2+d^2 = c^2+e^2$.  Adding the two relations gives $b^2 = e^2$.    Negating the last row and column if necessary, we may assume that $b=e$.   Then (considering that the first two rows are orthogonal) $d=-c$ and (from rows 1 and 4) $f=-a$.  But then rows 1 and 3 have inner product $ad+cf = -2ac \neq 0$, a contradiction.
\end{proof}

However, as the main result of this section shows below, this is the only possible even order with no such matrix.  (As is standard, we denote the square all-ones matrix by $J$.)

\begin{theorem} \label{theorem:symmetric}
There exists a symmetric $\OMZD(n)$ if and only if $n$ is even and $n\neq 4$.
\end{theorem}

\begin{proof} First, we suppose that $n$ is even and let $n=2m$.  Since there is a symmetric $\OMZD(2)$ (Example~\ref{example:conf}) but no symmetric $\OMZD(4)$ (Proposition~\ref{prop:symm4}), we assume that $m\geq 3$.

Define $A=J-I$ and $B=\alpha I+ \beta J$, where $\alpha=\sqrt{m^2-1}$ and
\[ \beta = \frac{-\sqrt{m^2-1}+\sqrt{2m-1}}{m}. \]
Now consider the matrix
\[ M = \left[ \begin{array}{c|c} A & B \\ \hline B & -A \end{array} \right]. \]
Clearly, $M$ is symmetric; we shall show that $M$ is an $\OMZD(2m)$.  Now,
\renewcommand{\arraystretch}{1.25}
\[ MM^{\mathrm{T}} = \left[ \begin{array}{c|c} A^2+B^2 & AB-BA \\ \hline BA-AB & A^2+B^2 \end{array} \right]; \]
\renewcommand{\arraystretch}{1.0}
by construction, $A$ and $B$ commute, so $AB-BA=BA-AB=0$, while
\begin{eqnarray*}
A^2+B^2 & = & (J-I)^2 + (\alpha I+ \beta J)^2 \\
        & = & (\alpha^2+1)I + (m\beta^2 +2\alpha\beta + m-2)J \\
				& = & m^2 I + (m\beta^2 +2\beta\sqrt{m^2-1} + m-2)J \\
				& = & m^2 I
\end{eqnarray*}
since our choice of $\beta$ is a solution to the quadratic equation $m\beta^2 +2\beta\sqrt{m^2-1} + m-2 = 0$.  Therefore, $M$ is orthogonal.

It remains to show that $M$ has the desired pattern of zero entries: in order for $M$ to be an $\OMZD(2m)$, it suffices to show that $B$ has no zero entries.  Now, the off-diagonal entries of $B$ are all $\beta$, and $\beta=0$ if and only if $m^2-1 = 2m-1$, which happens precisely when $m=2$.  The diagonal entries of $B$ are all $\alpha+\beta$, and we have that $\alpha+\beta=0$ if and only if $(m-1)\sqrt{m^2-1} + \sqrt{2m-1} =0$; since $m\geq 3$ this is impossible.\footnote{If we had chosen the other solution to the quadratic equation in $\beta$, we would always have $\beta\neq 0$, but would have $\alpha+\beta=0$ when $m=2$ instead. Choosing $\alpha=-\sqrt{m^2-1}$ does not yield a solution for $\beta$ which works for $m=2$ either.}  Thus we have a symmetric $\OMZD(n)$ for any even $n$ (except $n=4$).

Now suppose that $n$ is odd, and for a contradiction suppose that $M$ is a symmetric $\OMZD(n)$.  Since $M$ is symmetric, its eigenvalues are real.  Without loss of generality, we may assume that $MM^\mathrm{T}=I_n$, so therefore its eigenvalues are $\pm 1$.  Since $M$ has trace $0$, it follows that the eigenvalues $1$ and $-1$ have the same multiplicity.  But since $n$ is odd, this is impossible.  Hence, by contradiction, a symmetric $\OMZD(n)$ does not exist when $n$ is odd.
\end{proof}

As a symmetric $\OMZD(4)$ does not exist, and conference matrices only yield symmetric $\OMZD(n)$ for $n\equiv 2$~mod~$4$, the first case where Theorem~\ref{theorem:symmetric} needs to be applied is to obtain a symmetric $\OMZD(8)$.

\begin{example} \label{example:symmetric8} \rm
We consider the construction of Theorem~\ref{theorem:symmetric} for the case $n=2m=8$, so $m=4$, $\alpha = \sqrt{15}$ and $\beta = \frac{\sqrt{7}-\sqrt{15}}{4}$.  Then we obtain the following symmetric $\OMZD(8)$:
\renewcommand{\arraystretch}{1.2}
\[ \left[ \begin{array}{cccc|cccc}
0 & 1 & 1 & 1 & \alpha+\beta & \beta & \beta & \beta \\
1 & 0 & 1 & 1 & \beta & \alpha+\beta & \beta & \beta \\
1 & 1 & 0 & 1 & \beta & \beta & \alpha+\beta & \beta \\
1 & 1 & 1 & 0 & \beta & \beta & \beta & \alpha+\beta \\
\hline
\alpha+\beta & \beta & \beta & \beta & 0 & -1 & -1 & -1 \\
\beta & \alpha+\beta & \beta & \beta & -1 & 0 & -1 & -1 \\
\beta & \beta & \alpha+\beta & \beta & -1 & -1 & 0 & -1 \\
\beta & \beta & \beta & \alpha+\beta & -1 & -1 & -1 & 0 
\end{array} \right]. \]
\renewcommand{\arraystretch}{1.0}
\end{example}

\section{A construction using doubly regular tournaments} \label{sec:drt}

In this section, we present an alternative construction of OMZDs for certain orders congruent to $3$ modulo $4$.  This differs from our earlier approach by being more combinatorial in nature, making use of the following class of objects.

\begin{defn} \label{defn:drt}
A {\em doubly regular tournament} of order $q$, denoted $\DRT(q)$, is an orientation of the complete graph on $q$ vertices, where each vertex has out-degree $k$, and for any pair of distinct vertices $u,w$ there are exactly $\lambda$ vertices with arcs from both $u$ and $w$.
\end{defn}

If follows from the definition that $k=(q-1)/2$ and $\lambda=(q-3)/4$, and thus that $q\equiv 3$~mod~$4$.  Furthermore, if $A$ is the adjacency matrix of a $\DRT(q)$, then $A+A^\mathrm{T}=J-I$ and 
\[ AA^\mathrm{T} = \frac{q+1}{4}I + \frac{q-3}{4}J. \]
It was shown by Reid and Brown in 1972~\cite{ReidBrown72} that a $\DRT(q)$ is equivalent to a {\em skew-Hadamard matrix} of order $q+1$: this is a matrix $H$ with entries $\pm 1$ which satisfies $HH^\mathrm{T}=(q+1)I$ and $H+H^\mathrm{T}=2I$.  We note that $H$ is a skew-Hadamard matrix if and only if $H-I$ is a skew-symmetric conference matrix.  It also follows that if $A$ is the adjacency matrix of a $\DRT(q)$, then it is the incidence matrix of a {\em skew-Hadamard $2$-design}; the relationship between Hadamard matrices and symmetric designs is discussed in detail in~\cite[Chapter~4]{Ionin06}.

The relevance of doubly regular tournaments to this paper is demonstrated in the next result.

\begin{theorem} \label{theorem:drt}
Suppose that $q\neq 3$ and there exists a doubly regular tournament of order $q$.  Then there exists an $\OMZD(q)$.
\end{theorem}

\begin{proof} Suppose that there exists a doubly regular tournament of order $q$ with adjacency matrix $A$.  Let $M=\alpha A + J-I$: our aim is to choose a suitable $\alpha$ so that $M$ is an $\OMZD(q)$.  First, we note that $M$ has zero diagonal; the off-diagonal entries are either $1$ or $1+\alpha$, so provided $\alpha\neq -1$ these are all nonzero.  So we choose $\alpha$ such that $MM^\mathrm{T}=cI$ for some $c$.  Since $A$ has $(q-1)/2$ entries of $1$ in each row and column, it follows that
\[ AJ = JA^\mathrm{T} = \frac{q-1}{2}J. \]
From this, we obtain
\begin{eqnarray*}
MM^\mathrm{T} & = & (\alpha A +J-I)(\alpha A^\mathrm{T} +J-I) \\
     & = & \alpha^2 AA^\mathrm{T} + \alpha(q-1)J - \alpha(A+A^\mathrm{T})+ (q-2)J +I\\
		 & = & \alpha^2\left( \frac{q+1}{4}I + \frac{q-3}{4}J \right) +\alpha(q-1)J-\alpha(J-I)+(q-2)J + I \\
		 & = & \left( \alpha^2 \frac{q+1}{4}-\alpha + 1\right)I + \left(\alpha^2 \frac{q-3}{4} + \alpha(q-2) + q-2 \right)J.
\end{eqnarray*}
By the quadratic formula, the coefficient of $J$ in the above will be zero precisely when
\[ \alpha = \frac{-2}{q-3}\left( q-2 \pm \sqrt{q-2} \right). \]
Since $q\equiv 3$~mod~$4$ and $q\neq 3$ we have that $q\geq 7$, so these values of $\alpha$ are real numbers.  Hence the matrix $M=\alpha A+J-I$ is an $\OMZD(q)$.
\end{proof}

There are several constructions of doubly regular tournaments in the literature, usually presented in terms of skew-Hadamard matrices: see the survey by Koukouvinos and Stylianou~\cite{Koukouvinos08} for details.  A conjecture of Seberry Wallis~\cite{Seberry71} (in terms of skew-Hadamard matrices) asserts that a $\DRT(q)$ exists if and only if $q\equiv 3$~mod~$4$; the smallest value for which no example is known is $q=275$ (see~\cite{Koukouvinos08}).  Two important constructions are in the following lemmas.

\begin{lemma} \label{lemma:paley}
Let $q$ be a prime power such that $q\equiv 3$~mod~$4$.  Then there exists a doubly regular tournament of order $q$.
\end{lemma}

Because of the equivalence between skew-symmetric conference matrices and skew-\linebreak Hadamard matrices, Lemma~\ref{lemma:paley} is implied by Theorem~\ref{theorem:paleyconf}, and thus by Paley's 1933 \linebreak paper \cite{Paley33}.

\begin{lemma}[Seberry Wallis~\cite{Seberry71a}; Reid and Brown~\cite{ReidBrown72}] \label{lemma:doubling}
Suppose that there exists a doubly regular tournament of order $q$.  Then there exists a doubly regular tournament of order $2q+1$.
\end{lemma}

By applying Theorem~\ref{theorem:drt} to the doubly regular tournaments obtained from Lemma~\ref{lemma:paley}, then recursively applying Lemma~\ref{lemma:doubling} to that, we obtain the following result.

\begin{corollary} \label{corollary:paley}
Let $q$ be a prime power such that $q\equiv 3$~mod~$4$ and $q\geq 7$.  Then for any $t\geq 0$, there exists an $\OMZD(2^t(q+1)-1)$.
\end{corollary}

\begin{example} \label{example:fano} \rm
We show how to obtain an $\OMZD(7)$ using the approach of Theorem~\ref{theorem:drt}.  The following matrix is the adjacency matrix of the unique doubly regular tournament of order~$7$:
\[ A = \begin{bmatrix}
0 & 1 & 1 & 0 & 1 & 0 & 0 \\
0 & 0 & 1 & 1 & 0 & 1 & 0 \\
0 & 0 & 0 & 1 & 1 & 0 & 1 \\
1 & 0 & 0 & 0 & 1 & 1 & 0 \\
0 & 1 & 0 & 0 & 0 & 1 & 1 \\
1 & 0 & 1 & 0 & 0 & 0 & 1 \\
1 & 1 & 0 & 1 & 0 & 0 & 0 
\end{bmatrix}. \]
(We remark that the corresponding skew-Hadamard $2$-design is the Fano plane.)
Now let $\alpha=-\frac{1}{2}(5-\sqrt{5})$.  Then the matrix $M$ below is an $\OMZD(7)$:
\[ M = \alpha A + J-I = \begin{bmatrix}
0 & \alpha+1 & \alpha+1 & 1 & \alpha+1 & 1 & 1 \\
1 & 0 & \alpha+1 & \alpha+1 & 1 & \alpha+1 & 1 \\
1 & 1 & 0 & \alpha+1 & \alpha+1 & 1 & \alpha+1 \\
\alpha+1 & 1 & 1 & 0 & \alpha+1 & \alpha+1 & 1 \\
1 & \alpha+1 & 1 & 1 & 0 & \alpha+1 & \alpha+1 \\
\alpha+1 & 1 & \alpha+1 & 1 & 1 & 0 & \alpha+1 \\
\alpha+1 & \alpha+1 & 1 & \alpha+1 & 1 & 1 & 0 
\end{bmatrix} \]
(one can verify that $MM^\mathrm{T}=cI$, where $c=(27-9\sqrt{5})/2$).
\end{example}

By applying the ``doubling construction'' of Lemma~\ref{lemma:doubling}, we can also obtain an $\OMZD(15)$ starting from the same matrix~$A$ as in Example~\ref{example:fano}.

\section{Application: distinct eigenvalues of graphs} \label{sec:eigenvalues}

Suppose $A$ is an $n\times n$ real symmetric matrix.  The {\em graph of $A$} is the simple graph on $n$ vertices labelled $1,\ldots,n$, with $i$ adjacent to $j$ if and only if $A_{ij}\neq 0$ (for $i\neq j$).  Conversely, for any simple graph $G$ with $n$ vertices, one can associate with $G$ the class of $n\times n$ matrices $\mathcal{S}(G)$, consisting of all real symmetric matrices whose graph is $G$.  We note that the definition of the graph of $A$ ignores its diagonal entries, so the diagonal entries of any matrix in $\mathcal{S}(G)$ are arbitrary.

Much recent research has been devoted to such classes of matrices (see, for instance, the surveys by Fallat and Hogben~\cite{FallatHogben1,FallatHogben2}).  In particular, in \cite{dmrg,Barrett,Duarte02,Fonseca10,KimShader09} the following question is considered: for a given graph $G$, what is the minimum number of distinct eigenvalues of a matrix in $\mathcal{S}(G)$?  This parameter is denoted $q(G)$.

It is not difficult to see that $q(G)=1$ if and only if the graph $G$ has no edges (see \cite[Lemma 2.1]{dmrg}).  So a natural question is to ask which graphs $G$ have $q(G)=2$.  Various results on this question were obtained by Ahmadi {\em et al.}\ in \cite{dmrg}: for instance, they showed that complete graphs $K_n$, complete bipartite graphs $K_{n,n}$ and hypercubes $Q_n$ all have $q(G)=2$.  As observed in \cite[Section~4]{dmrg}, the property that $q(G)=2$ for a non-null graph $G$ is equivalent to the existence of an orthogonal matrix in $\mathcal{S}(G)$.

The parameter $q(G)$ is not necessarily well-behaved when it comes to deleting edges from a graph $G$.  For instance, for a cycle $C_n$ on $n$ vertices we have $q(C_n) = \lceil n/2 \rceil$ (see \cite[Lemma 2.7]{dmrg}); if we delete an edge to obtain the path $P_n$ on $n$ vertices, we then have $q(P_n)=n$ (see \cite[Theorem 3.1]{Fonseca10}).  However, the deletion of an edge can also cause the minimum number of distinct eigenvalues of a graph to decrease: compare the examples in \cite[Figs.\ 6.1, 6.2]{dmrg}).  In fact, the difference between $q(G)$ and $q(G-e)$ can be made arbitrarily large in either direction.

\subsection{Bipartite graphs} \label{sec:bipartite}

For a bipartite graph $G$ with bipartition $X\cup Y$, where $X=\{1,\ldots,m\}$ and $Y=\{1',\ldots,n'\}$, we let $\mathcal{B}(G)$ be the class of $m\times n$ real matrices $B$, whose rows and columns are indexed by $X$ and $Y$ respectively, and where $B_{ij}\neq 0$ if and only if $\{i,j'\}$ is an edge of $G$.  Clearly, if $B\in \mathcal{B}(G)$, then the matrix%
\renewcommand{\arraystretch}{1.25}%
\[ A = \left[ \begin{array}{c|c} 0 & B \\ \hline B^\mathrm{T} & 0 \end{array} \right] \]%
\renewcommand{\arraystretch}{1.0}%
is in $\mathcal{S}(G)$.  In \cite[Section~6]{dmrg}, Ahmadi {\em et al.}\ showed the following, which provides a connection with orthogonal matrices.

\begin{theorem}[Ahmadi {\em et al.}\ \cite{dmrg}] \label{theorem:dmrg}
Suppose $G$ is a bipartite graph with bipartition $X\cup Y$.  Then \mbox{$q(G)=2$} if and only if $|X|=|Y|$ and there exists an orthogonal matrix $B\in \mathcal{B}(G)$.
\end{theorem}

A straightforward consequence of this theorem is the following.

\begin{corollary}[Ahmadi {\em et al.}\ \cite{dmrg}] \label{corollary:Knn}
For the complete bipartite graph $K_{n,n}$, $q(K_{n,n})=2$.  
\end{corollary}

\begin{proof}  Since $\mathcal{B}(K_{n,n})$ consists of all $n\times n$ matrices with no zero entries, all one requires is an orthogonal matrix with no zeroes; for $n\leq 2$ this is trivial, while for $n\geq 3$, the matrix $I-\frac{2}{n}J$ is orthogonal. \end{proof}

The main result of this subsection gives an application of OMZDs to the following class of graphs.  We denote by $G_n$ the bipartite graph obtained by deleting a perfect matching from $K_{n,n}$.

\begin{theorem} \label{theorem:application}
Where $G_n$ is as above, $q(G_n)=2$ unless $n=1$ or $n=3$.
\end{theorem}

\begin{proof} Suppose that $G_n$ has bipartition $X\cup Y$, where $X=\{1,\ldots,n\}$ and $Y=\{1',\ldots,n'\}$, and that the perfect matching that was deleted was the canonical one, i.e.\ $\{i,i'\}$ for $1\leq i \leq n$.  Then any matrix in $\mathcal{B}(G_n)$ has zero diagonal and nonzero entries elsewhere, so any orthogonal matrix in $\mathcal{B}(G_n)$ is an $\OMZD(n)$.  By Theorem~\ref{theorem:existence}, such a matrix exists for all $n\neq 1,\ 3$.  The result then follows from Theorem~\ref{theorem:dmrg}.
\end{proof}

Note that the graph $G_1$ has no edges and thus $q(G_1)=1$, while the graph $G_3$ is a cycle on six vertices $C_6$, and the non-existence of an $\OMZD(3)$ corresponds to the known result that $q(C_6)=3$ (cf.\ \cite[Lemma 2.7]{dmrg}).  

By Corollary~\ref{corollary:Knn}, we know that $q(K_{n,n})=2$; thus Theorem~\ref{theorem:application} gives examples of graphs where deleting a set of edges does not change the value of $q(G)$.  However, to obtain $G_n$ from $K_{n,n}$ one must delete several edges simultaneously; a natural question is to ask what happens to $q(G)$ if one deletes these edges sequentially, one at a time.  In this situation, to show that $q(G)$ remains equal to $2$ each time an edge is deleted, the following class of matrices may be used.

\subsection{Orthogonal matrices with partially-zero diagonal} \label{sec:partially-zero}

We define an {\em orthogonal matrix with partially-zero diagonal}, $\OMPZD(n,k)$, to be an $n\times n$ orthogonal matrix with exactly $k$ zero entries, all of which are on the main diagonal.  Clearly an $\OMZD(n)$ is an $\OMPZD(n,n)$.  On the other hand, an $\OMPZD(n,0)$ is an orthogonal matrix with no zero entries, which we observed in Corollary~\ref{corollary:Knn} exists for all $n\geq 1$.  The connection with bipartite graphs is given by the following result, whose proof is essentially the same as that of Theorem~\ref{theorem:application}.

\begin{theorem} \label{theorem:arbitrarymatching}
Let $G_{n,k}$ denote the bipartite graph obtained by deleting a matching of size $k$ from $K_{n,n}$.  Then $q(G_{n,k})=2$ if and only if there exists an $\OMPZD(n,k)$.
\end{theorem}

So it remains to show how to construct OMPZDs.  One can easily show that there is no $\OMPZD(2,1)$ or $\OMPZD(3,2)$, while Theorem~\ref{theorem:existence} shows that there is no $\OMPZD(1,1)$ or $\OMPZD(3,3)$.  Some small examples are given below.

\begin{example}\label{example:small-ompzd} \rm
The following matrices are an $\OMPZD(3,1)$, an $\OMPZD(4,3)$ and an $\OMPZD(5,4)$:
\[ \begin{bmatrix} 1        & 1         & \sqrt{2} \\
                   1        & 1         & -\sqrt{2} \\
									 \sqrt{2} & -\sqrt{2} & 0 \end{bmatrix}
\qquad
\begin{bmatrix} 1 & 1      & 1      & 1 \\
                1 & 0      & \alpha & \beta \\
                1 & \beta  & 0      & \alpha \\
                1 & \alpha & \beta  & 0 \end{bmatrix}, 
\qquad 															
\begin{bmatrix} 1 & 1       & 1       & 1       & 1 \\
                1 & 0       & \varphi & \beta   & \psi \\
                1 & \psi    & 0       & \varphi & \beta \\
                1 & \beta   & \psi    & 0       & \varphi \\
                1 & \varphi & \beta   & \psi    & 0 \end{bmatrix} \]
(where $\alpha= \frac{-1+\sqrt{5}}{2}$, $\beta = \frac{-1-\sqrt{5}}{2}$, and $\psi,\varphi$ are the roots of $2x^2+(1-\sqrt{5})x-1=0$).
\end{example}

In fact, most OMPZDs can be obtained from OMZDs, using the approach of Cheon {\em et al.}~\cite{Cheon99} for reducing the number of zero entries in an orthogonal matrix, which gives the following result.

\begin{lemma} \label{lemma:killingzeroes}
Suppose that $n\geq 4$.  Then, for any $k$ such that $0\leq k \leq n-2$, there exists an $\OMPZD(n,k)$.
\end{lemma}

\begin{proof} Let $R(\theta)$ denote the following matrix (where $-\pi < \theta \leq \pi$):
\[ R(\theta)=\left[ \begin{array}{c|c}
             \begin{array}{cc}
						 \cos\theta & -\sin\theta \\
						 \sin\theta &  \cos\theta \end{array} & 0 \\ 
						\hline  0 & I_{n-2} \end{array} \right]. \]
It is clear that $R(\theta)$ is an orthogonal matrix.  Furthermore, for any orthogonal matrix $A$, the product $AR(\theta)$ must also be orthogonal.  Also, the columns of $AR(\theta)$ are identitical to those of $A=(a_{ij})$ apart from the first two, which have the form
\[ \begin{bmatrix} a_{1,1}\cos\theta + a_{1,2}\sin\theta \\
                   \vdots \\
                   a_{n,1}\cos\theta + a_{n,2}\sin\theta \end{bmatrix} 
\qquad \textnormal{and} \qquad
  \begin{bmatrix} -a_{1,1}\sin\theta + a_{1,2}\cos\theta \\
                  \vdots \\
                  -a_{n,1}\sin\theta + a_{n,2}\cos\theta \end{bmatrix}. \]
Let $A$ be an $\OMZD(n)$, which exists by Theorem~\ref{theorem:existence} since $n\geq 4$.  Then we may choose an arbitrarily small value of $\theta>0$ to ensure that all entries of these two columns are nonzero, and thus $AR(\theta)$ is an $\OMPZD(n,n-2)$.  By permuting rows and columns, and repeatedly applying this process, we may obtain an $\OMPZD(n,k)$ for any $k$ where $0\leq k \leq n-2$. 
\end{proof}

We remark that this approach does not work to obtain an $\OMPZD(n,n-1)$.  However, our methods for constructing OMZDs from Section~\ref{sec:existence} can easily be adapted to this situation.

\begin{lemma} \label{lemma:ompzd}
There exists an $\OMPZD(n,n-1)$ if and only if $n\neq 2,3$.
\end{lemma}

\begin{proof} An $\OMPZD(1,0)$ is trivial to construct.  As already observed, neither an $\OMPZD(2,1)$ nor an $\OMPZD(3,2)$ can exist.  Examples of an $\OMPZD(4,3)$ and an $\OMPZD(5,4)$ were given in Example~\ref{example:small-ompzd} above.  So we assume that $n\geq 6$.  Let $M$ be an $\OMPZD(4,3)$, scaled so that $MM^\mathrm{T}=I_4$, and arranged so that the $(1,1)$ entry is zero.  Let $N$ be an $\OMZD(n-2)$, which exists by Theorem~\ref{theorem:existence}, also scaled such that $NN^\mathrm{T}=I_{n-2}$.  As in the proof of Lemma~\ref{lemma:adding}, we suppose that these have the form
\renewcommand{\arraystretch}{1.25}
\[ M = \left[ \begin{array}{c|c} 0 & u^\mathrm{T} \\ \hline v & B \end{array} \right], \qquad N = \left[ \begin{array}{c|c} 0 & x^\mathrm{T} \\ \hline y & C \end{array} \right],
\]
and then define
\[ Q = \left[ \begin{array}{c|c} B & vx^\mathrm{T} \\ \hline yu^\mathrm{T} & C \end{array} \right]. \]
\renewcommand{\arraystretch}{1.0}%
It follows---using the same method as Lemma~\ref{lemma:adding}---that $Q$ is an $\OMPZD(n,n-1)$.
\end{proof}

Combining Theorem~\ref{theorem:existence}, Example~\ref{example:small-ompzd}, Lemmas~\ref{lemma:killingzeroes} and~\ref{lemma:ompzd}, and the existence of $\OMPZD(n,0)$ from Corollary~\ref{corollary:Knn}, we have proved the following.

\begin{theorem} \label{theorem:ompzd}
Let $n\geq 1$ and $0\leq k \leq n$.  Then there exists an $\OMPZD(n,k)$ if and only if $(n,k)\not\in \left\{ (1,1),\, (2,1),\, (3,2),\, (3,3) \right\}$.
\end{theorem}

Immediately, we have the following result for graphs.

\begin{corollary} \label{corollary:qG}
Suppose that $n\geq 1$ and $1\leq k\leq n$, and let $G_{n,k}$ be the bipartite graph obtained by deleting a matching of size $k$ from $K_{n,n}$.  Then, unless $(n,k)\in \left\{ (1,1),\, (2,1),\, (3,2),\, (3,3) \right\}$, we have $q(G_{n,k})=2$.
\end{corollary}

The two exceptions which we have not already encountered are when $(n,k)=(2,1)$ and $(n,k)=(3,2)$.  In the former case, the graph $G_{2,1}$ is a path with $4$ vertices, so $q(G_{2,1})=4$ by \cite[Theorem 3.1]{Fonseca10}. In the latter case, the graph $G_{3,2}$ is a Cartesian product $P_3\ \Box\ K_2$; by~\cite[Theorems~3.2 and 6.7]{dmrg} we have $3\leq q(G_{3,2})\leq 4$.

\subsection{Complete multipartite graphs} \label{sec:multi}
The existence of symmetric OMZDs of even order can be used to demonstrate that that $q(G)=2$ for yet another family of graphs.  We use $K_{n^m}$ to denote the complete multipartite graph with $m$ parts of size $n$.

\begin{theorem} \label{theorem:multi}
Let $K_{n^m}$ be the complete multipartite graph with $m$ parts of size $n$, where $m$ is even and $m\neq 4$.  Then $q(K_{n^m})=2$.
\end{theorem}

\begin{proof} Since $m$ is even, we have $m>1$ which ensures that the graph is non-null, and thus $q(K_{n^m})>1$.  Also, since an orthogonal matrix has exactly two distinct eigenvalues, it suffices to demonstrate the existence of an orthogonal matrix in $\mathcal{S}(K_{n^m})$.  Let $A=(a_{ij})$ be a symmetric $\OMZD(m)$, which (since $m$ is even and $m\neq 4$) exists by Theorem~\ref{theorem:symmetric}, and let $B$ be a nowhere-zero symmetric orthogonal matrix of order $n$ (as observed earlier, this is trivial to find for $n\leq 2$, and $B=I-\frac{2}{n}J$ works for $n\geq 3$).  Now consider the Kronecker product of $A$ and $B$, which has the form
\[ A\otimes B = \begin{bmatrix} 0       & a_{12}B & \cdots & a_{1m}B \\
                                a_{21}B & 0       & \cdots & a_{2m}B \\
																\vdots  &         & \ddots & \vdots \\
																a_{m1}B & a_{m2}B & \cdots & 0 \end{bmatrix}. \]
By standard properties of Kronecker products, we have that $A\otimes B$ is symmetric and orthogonal.  Also, the all-zero blocks on the diagonal and nowhere-zero blocks off the diagonal show that $A\otimes B \in \mathcal{S}(K_{n^m})$.  Hence, when $m$ is even and $m\neq 4$, $q(K_{n^m})=2$. \end{proof}

In the case $m=2$, we have a complete bipartite graph $K_{n,n}$, and the matrix obtained here is exactly the matrix described in~\cite[Corollary 6.5]{dmrg}.  We note that, as matrices in $\mathcal{S}(G)$ can have arbitrary diagonal entries, the matrices in the proof of Theorem~\ref{theorem:multi} have more structure than is necessary.  Also, the assumptions that $m$ is even and $m\neq 4$ do not appear to be natural; we suspect that for any complete multipartite graph $K_{n^m}$ where $n\geq 1$ and $m\geq 2$, we will always have that $q(K_{n^m})=2$.

\section*{Acknowledgements}
The first author is supported by an NSERC Discovery Grant and a Memorial University of Newfoundland startup grant, and would like to thank the Discrete Mathematics Research Group at the University of Regina for sparking his interest in this problem.  Both authors would like to thank the organizers of the 2016 Western Canada Linear Algebra Meeting and 2016 Prairie Discrete Mathematics Workshop at the University of Manitoba, which facilitated this collaboration.

\end{document}